\documentclass{article}
\usepackage{amssymb}
\usepackage{amsmath}

\newtheorem{theorem}{Theorem}

\newtheorem{corollary}[theorem]{Corollary}

\newtheorem{remark}[theorem]{Remark}

\newenvironment{proof}[1][Proof]{\noindent\textbf{#1.} }{\ \rule{0.5em}{0.5em}}
\input{tcilatex}
\begin{document}

\begin{center}
\bigskip

{\Large \textsc{Two Distinct Seasonally Fractionally Differenced Periodic
Processes}} \bigskip

Ahmed BENSALMA $^{1,2}$

\textit{$^{\emph{1}}$ }ENSSEA P\^{o}le universitaire Kol\'{e}a, Tipaza (ex:
INPS 11 Doudou mokhtar Benaknoun) Algeria)

\textit{$^{\emph{2}}$}Faculty of Mathematics, University of Science and
Technology Houari Boumediene, Algeria
\end{center}

\abstract{This article is devoted to study the effects of the S-periodical fractional differencing filter $(1-L^{S})^{D_{t}}$. To put this effect in evidence, we have derived the periodic auto-covariance functions of two distinct univariate seasonally fractionally differenced periodic models. A multivariate representation of periodically correlated process is exploited to provide the exact and approximated expression auto-covariance of each models. The distinction between the models is clearly obvious through the expression of periodic auto-covariance function. Besides producing different auto-covariance functions, the two \ models differ in their implications. In the first model, the seasons of the multivariate series are separately fractionally integrated. In the second model, however, the seasons for the univariate series are fractionally co-integrated. On the simulated sample, for each models, with the same parameters, the empirical periodic
auto-covariance are calculated and graphically represented for illustrating the results and support the comparison between the two models.}

\section{Introduction}

Since their introduction by Gladyshev $(1961,1963)$ much attention has been
given to periodically correlated processes. The interest, for such processes
is due to their potential use in modeling of cyclical phenomena appearing in
hydrology, climatology and in econometrics. Following pioneer work of
Gladyshev $(1963)$, an important part of the literature has been devoted to
the periodically correlated discrete time processes. A discrete time process
is periodically correlated, if there is a non zero integer $S$ such that%
\begin{equation*}
E\left( X_{t+S}\right) =E\left( X_{t}\right) \text{ and }%
Cov(X_{t_{1}+S},X_{t_{2}+S})=Cov(X_{t_{1}},X_{t_{2}}).
\end{equation*}

A review of the periodically correlated discrete time processes is proposed
in Lund and Basawa $(1999)$, Bentarzi and Hallin $(1994)$ give invertibility
conditions for periodic moving average.. A large part of the literature on
the subject is devoted to the periodic $ARMA$ ($PARMA$) models, which have
the following representation:%
\begin{equation*}
X_{t}-\underset{i=0}{\overset{p_{t}}{\sum }}\phi _{i,t}X_{t-i}=\underset{j=0}%
{\overset{q_{t}}{\sum }}\theta _{j,t}u_{t-j}\text{, \ }t=0,\pm 1,\pm 2,....,
\end{equation*}%
where $u_{t}$ is a zero-mean white noise with variance $\sigma _{t}^{2}$.
Among searchers who were interested with the periodic autoregressifs
processes not periodically stationary, we cite Boswijk and Franses $(1995)$
which studied the problem of the presence of a unit root in a periodic
autoregression model of order $p$ $(PAR(p))$ and Boswijk, Franses and
Haldrup $(1997)$ which studied the presence of multiple unit roots in a
periodic autoregression model of order $p$. All work cited above were made
under the assumption that the processes are periodically integrated of order
zero $(PI(0))$, integrated of order one $(I(1))$ or periodically integrated
of order one $((PI(1)$). However currently, it well-known that in the
scientific fields mentioned above (hydrology, meteorology, econometrics)
much of sets of data that have a certain periodicity; have also a long range
dependence (or long memory). Such phenomena can be modeled by stationary
processes. The stationary processes with seasonal long memory are well know
(see for example Gray, Zhang and Woodward $(1989)$: Garma models;
Purter-Hudak $(1990)$:Seasonal $ARFIMA$; Oppenheim, G. and al $(2000)$; ould
Haye and al $(2003)$ for references, properties and simulations). Another
alternative, to take account of certain periodic phenomena with long memory
is to consider nonstationary models (but periodically stationary) such as
the periodically correlated processes with long memory. The periodically
correlated processes, within the meaning of Gladyshev $(1963)$, with long
memory did not receive much attention on behalf of the statisticians and the
probabilists. Among works associating periodicity within the meaning of
Gladyshev $(1963)$ and the presence of long memory we cite, Hui and Li $%
(1995)$, Franses and Ooms $(1997)$, Ooms and Franses $(2001).$

For modelling of the Hong Kong United Christian Hospital attendance series,
Hui and Li $(1995)$ propose a $2$-periodic correlated process,%
\begin{equation}
(1-L)^{d_{t}}Y_{t}=u_{t},  \tag{$1.1$}
\end{equation}%
where $\left\{ u_{t},t\in 
\mathbb{Z}
\right\} $ is a zero mean white noise with variance $\sigma _{t}^{2}$, and $%
d_{t}$ the $2$-periodic fractional parameter. The empirical series $y_{t}$
concerns seventy five (approximately one and half years) data on the average
number of people entering the emergency unit on weekday and weekend.

On the other hand, in order to analyzes the long-memory properties in the
conditional mean of the quarterly inflation rate in the United Kingdom
Franses and Ooms $(1997)$ propose a $4$-periodic correlated process,%
\begin{equation}
Y_{t}=(1-L)^{-d_{t}}u_{t},  \tag{$1.2$}
\end{equation}%
where $\left\{ u_{t},t\in 
\mathbb{Z}
\right\} $ is defined as above and $d_{t}$ is the $4$-periodic fractional
parameter.

Finally, for the monthly empirical data, which concern the log transformed
data of the monthly mean river flow in cubic feet per second, Ooms and
Franses $(2001)$ propose to use the seasonal periodic fractional operator
defined, in simple framework as follows,%
\begin{equation}
Y_{t}=(1-L^{S})^{-D_{t}}u_{t},\text{ }S=12  \tag{$1.3$}
\end{equation}%
where $\left\{ u_{t},t\in 
\mathbb{Z}
\right\} $ is defined as above and $S=12$.

The main difference between, the one hand, the models $(1.1)$ and $(1.2)$
and the other hand, the model $(1.3)$, is in the unit of lag to which the
fractional difference operator is applied. In the models $(1.1)$ and $(1.2)$
the fractional difference operator was applied to weekly and quarterly lags,
respectively, corresponding to the basic time interval of the time series
analyzed. In the model $(1.3)$ the fractional difference operator was
applied to yearly, which is the seasonal lag of the time series analyzed.
Indeed, by using a binomial expansion for the difference operator $%
(1-L)^{d_{t}}$, $(1-L)^{-d_{t}}$, $(1-L^{S})^{-d_{t}}$ we can rewrite,
respectively, the models $(1.1)$, $(1.2)$ and $(1.3)$ as the following,%
\begin{equation}
\sum_{j=0}^{\infty }\frac{\Gamma \left( j-d_{t}\right) }{\Gamma \left(
j+1\right) \Gamma \left( -d_{t}\right) }Y_{t-j}=u_{t},  \tag{$1.4$}
\end{equation}%
\begin{equation}
Y_{t}=\sum_{j=0}^{\infty }\frac{\Gamma \left( j+d_{t}\right) }{\Gamma \left(
j+1\right) \Gamma \left( d_{t}\right) }u_{t-j},  \tag{$1.5$}
\end{equation}%
\begin{equation}
Y_{t}=\sum_{j=0}^{\infty }\frac{\Gamma \left( j+D_{t}\right) }{\Gamma \left(
j+1\right) \Gamma \left( D_{t}\right) }u_{t-Sj}.  \tag{$1.6$}
\end{equation}

where%
\begin{equation*}
\Gamma \left( z\right) =\left\{ 
\begin{array}{ll}
\int_{0}^{+\infty }s^{z-1}e^{-z}ds, & \text{if }z>0 \\ 
\infty & \text{if }z=0,%
\end{array}%
\right.
\end{equation*}%
if $z<0$, $\Gamma \left( z\right) $ is defined in terms of the above
expressions and the recurrence formula $z\Gamma \left( z\right) =\Gamma
\left( z+1\right) $.

While, the invertibility and stationarity conditions of the model $(1.3)$
are known (see Ooms and Franses $2001$), apart when $d_{t}=d$ is a constant,
nothing is clear about the models $(1.1)$ and $(1.2)$. More precisely, no
thing is clear about the stationarity conditions for the model $(1.1)$,
because his infinite moving average representation is unknown and no thing
is clear about the invertibility conditions for the model $(1.2)$, because
his infinite autoregressive representation is unknown. The model $(1.3)$ is
invertible and stationary if $-0.5<D_{t}<0.5$ and it is easy to show in this
case that the infinite autoregressive representation of the process $y_{t}$
is given by%
\begin{equation*}
\sum_{j=0}^{\infty }\frac{\Gamma \left( j-D_{t}\right) }{\Gamma \left(
j+1\right) \Gamma \left( -D_{t}\right) }Y_{t-Sj}=u_{t}.
\end{equation*}

For the model $(1.4)$, at any case, in general, we have,%
\begin{equation*}
Y_{t}\neq \sum_{j=0}^{\infty }\frac{\Gamma \left( j+d_{t}\right) }{\Gamma
\left( j+1\right) \Gamma \left( d_{t}\right) }u_{t-j},
\end{equation*}%
and for the model $(1.5)$, at any case, in general, we have,%
\begin{equation*}
\sum_{j=0}^{\infty }\frac{\Gamma \left( j-d_{t}\right) }{\Gamma \left(
j+1\right) \Gamma \left( -d_{t}\right) }Y_{t-j}\neq u_{t}.
\end{equation*}

For the particular periodic $ARFIMA(0,d_{t},0)$, namely $%
(1-L)^{d_{t}}y_{t}=u_{t},$ $u_{t}\sim i.i.d(0,\sigma _{t}^{2})$, the
infinite moving average representation is unknown. In this paper, we give
the closed form of this representation. It is important to known such
representation in order to deduce the stationarity condition of this type of
model. Unfortunately, the closed form obtained is not easy to handle due to
her parametric complexity (see Appendix).

Since the $PARFIMA(p,d_{t},q)$ is not easy to handle. The work that we
present in this article is concerned only on the Seasonal periodical
fractional operator, namely $(1-L^{S})^{D_{t}}$. More Precisely, in this
work we are interested in certain theoretical properties of the $%
SPARFIMA(p,0,0)(0,D_{t},0)_{S}$ (Seasonal periodic $ARFIMA$). The study of
the theoretical properties of this class of models remains to be made;
because among works which evoke this class, only one exists; that is of
Franses and Ooms $(2001)$. The work of Franses and Ooms has to consist in
adjusting a $SPARFIMA(p,0,0)(0,D_{t},0)_{S}$ to a set of real data.
Precisely the model considered by Ooms and Franses is defined as follows:

\begin{equation*}
\phi _{t}(L)\left( X_{t}-\mu _{t}\right) =\eta _{t},\text{ }t\in 
\mathbb{N}
^{\ast }\text{, with }\eta _{t}=(1-L^{S})^{-D_{t}}u_{t},
\end{equation*}%
where $\mu _{t}$ is S-periodical constant such as $\mu _{t}=\mu _{t+S}$, $%
\phi _{t}(L)=1-\phi _{t,1}L-\phi _{t,2}L^{2}-...-\phi _{t,p}L^{p}$. The
parameters $\phi _{t,i}$ $i=1,...,p$ are periodic functions in $t,$ and $%
\eta _{t}$ a white noise seasonally fractionally integrated of order $D_{t}$%
, where $D_{t}$ is S-periodical fractional parameter. The model above, if $%
0\leq D_{t}<0.5$, $\forall t$ can be written as follows:

\begin{equation}
(1-L^{S})^{D_{t}}\phi _{t}(L)\left( X_{t}-\mu _{t}\right) =u_{t},\text{ }%
t\in 
\mathbb{Z}
.  \tag{$Model(I)$}
\end{equation}

There is another class of models $SPARFIMA(p,0,0)(0,D_{t},0)_{S}$ distinct
from that used by Franses and Ooms $(2001)$; this class is defined as
follows:%
\begin{equation}
\phi _{t}(L)(1-L^{S})^{D_{t}}\left( X_{t}-\mu _{t}\right) =u_{t},\text{ }%
t\in 
\mathbb{Z}
,  \tag{$Model(II)$}
\end{equation}%
where $\mu _{t}$, $\phi _{t}(L)$, $D_{t}$ are defined like above. These two
classes coincide, only if $D_{t}=D,$ $\forall t$, since, generally, the
composition of $\phi _{t}(L)$ and $(1-L^{S})^{D_{t}}$ is not necessarily
commutative. To convince, it is sufficient to notice that the $S$-variate
representation of the model $(I)$ is a $VARFI$ model (vector autoregressive
model, driven by fractionally integrated innovation) whereas the
multivariate writing of the model $(II)$ is a $FIVAR$ model (fractionally
integrated vector autoregression) (see Rebecca Sela and Clifford Mr. Hurvich 
$(2008)$). These two distinct classes, generalize the univariate model $%
ARFIMA$, the first is closely related to the cointegrated processes, whereas
the second is closely related to the integrated processes. Consequently, in
our case, the model $(I)$ is closely related to the cointegrated season and
the model $(II)$ is closely related to the integrated season.

In order to distinguish between the model $(I)$ and $(II)$, we note them,
respectively as the following: $PAR(p)-PSFI(D_{t})$ and $PSFI(D_{t})-PAR(p)$%
. The rest of this paper is organized as follows: section $2$ is devoted to
defined two class of processes; the periodic autoregressive of order $p$
process with periodic seasonal fractional integrated of order $D_{t}$
innovation, namely $PAR(p)-PSFI(D_{t})$ and the periodic seasonal fractional
integrated process, periodic autoregressive of order $p$, namely $%
PSFI(D_{t})-PAR(p)$. In section $3$, for each model defined in section $2$,
we provide the exact \ and approximated expression of the periodic
autocovariances function. In the section $4$, on the simulated samples for
each model, with the same parameters for the model $(I)$ and $(II)$, the
empirical periodic autocovariances are calculated and graphically
represented for illustrating the theoretical results and comparison between
the two models.

Without restricting the generality, we suppose that all processes defined
below have zero mean.

\section{Representation and notation}

\subsection{S-periodical seasonally fractionally integrated, periodic
autoregressive process ($PSFI(D_{t})-PAR(p)$)}

A periodically correlated process $\left\{ Y_{t},t\in 
\mathbb{Z}
\right\} $ is said S-periodical seasonally fractionally integrated of order $%
D_{t}$, periodic autoregressive of order $p$; if it has the following
representation:%
\begin{equation}
\Phi _{t}(L)(1-L^{S})^{D_{t}}Y_{t}=u_{t},\text{ \ \ \ }t\in 
\mathbb{Z}
,  \tag{$2.1$}
\end{equation}%
where $\left\{ u_{t},t\in 
\mathbb{Z}
\right\} $ is a zero mean white noise with variance $\sigma _{t}^{2}$and $%
(1-L^{S})^{D_{t}}$ are defined like above. $\phi _{t}(L)=1-\phi _{t,1}L-\phi
_{t,2}L^{2}-...-\phi _{t,p}L^{p}$ where $\phi _{t,1},...,\phi _{t,p}$ are $S$%
-periodical parameters.\newline
Letting $\underline{\mathbf{Y}}_{\tau }=\left( \mathbf{Y}_{1,\tau },...,%
\mathbf{Y}_{s,\tau },...,\mathbf{Y}_{S,\tau }\right) ^{\prime }$ and $%
\underline{\mathbf{u}}_{\tau }=\left( \mathbf{u}_{1,\tau },...,\mathbf{u}%
_{s,\tau },...,\mathbf{u}_{S,\tau }\right) ^{\prime }$ with $\mathbf{Y}%
_{s,\tau }=Y_{s+S\tau }$ and $\mathbf{u}_{s,\tau }=u_{s+S\tau }$ then the
process $(2.1)$ can be rewritten in the $S$ variate form%
\begin{equation}
\mathbf{\Phi }_{0}\nabla _{S}^{\underline{\mathbf{D}}}(L)\underline{\mathbf{Y%
}}_{\tau }-\underset{i=1}{\overset{\mathbf{P}}{\sum }}\mathbf{\Phi }%
_{i}\nabla _{S}^{\underline{\mathbf{D}}}(L)\underline{\mathbf{Y}}_{\tau -i}=%
\underline{\mathbf{u}}_{\tau },  \tag{$2.2$}
\end{equation}%
where $\mathbf{P}=\left[ \frac{p+1}{S}\right] +1$, with $\left[ x\right] $
denotes the smallest integer than or equal to $x$, $\nabla _{S}^{\underline{%
\mathbf{D}}}(L)$ is defined like above. The autoregressive coefficient
matrices are given by%
\begin{equation*}
\left( \mathbf{\Phi }_{0}\right) _{s,j}=\left\{ 
\begin{array}{c}
1\text{ \ \ \ \ \ \ }s=j, \\ 
0\text{ \ \ \ \ \ \ }s<j, \\ 
-\phi _{s-j,j}\text{ \ \ \ \ \ }s>j,%
\end{array}%
\right.
\end{equation*}%
and%
\begin{equation*}
\left( \mathbf{\Phi }_{i}\right) _{s,j}=\phi _{iS+s-j,s},\text{ \ }%
s,j=1,...,S\text{ and }1\leq i\leq \mathbf{P}.
\end{equation*}

The periodic stationarity condition of the model $(2.2)$ is the same as the
stationarity condition of it equivalent fractional integrated vector
autoregression, namely $FIVAR$, (Rebecca Sela and Clifford Hurvich $(2008)$)
representation $(2.2)$, which means that the roots of the determinantal
equation%
\begin{equation*}
\det \left( I_{S}z^{\mathbf{P}}-\underset{i=1}{\overset{\mathbf{P}}{\sum }}%
\mathbf{\Phi }_{0}^{-1}\mathbf{\Phi }_{i}z^{\mathbf{P}-i}\right) =0,
\end{equation*}%
are less than $1$ in absolute value (Hannan $(1970)$, Fuller $(1976)$) and%
\begin{equation*}
0\leq D_{s}<0.5,\text{for all }s=1,...,S,
\end{equation*}%
(Hosking $(1981)$). If the process $(2.2)$ is stationary, then it has an
infinite moving average representation given by%
\begin{align}
\underline{\mathbf{Y}}_{\tau }& =\nabla _{S}^{\underline{\mathbf{D}}}(L)^{-1}%
\mathbf{\Phi }\left( L\right) ^{-1}\underline{\mathbf{u}}_{\tau },  \notag \\
& =\left( \underset{j=0}{\overset{\infty }{\tsum }}\Psi _{j}L^{j}\right)
\left( \underset{j=0}{\overset{\infty }{\tsum }}\Pi _{j}L^{j}\right) 
\underline{\mathbf{u}}_{\tau },  \notag \\
& =\underset{j=0}{\overset{\infty }{\tsum }}\left( \underset{k=0}{\overset{j}%
{\tsum }}\Psi _{k}\Pi _{j-k}\right) \underline{\mathbf{u}}_{\tau -j},  \notag
\\
& =\underset{j=0}{\overset{\infty }{\tsum }}C_{j}\underline{\mathbf{u}}%
_{\tau -j},  \tag{$2.3$}
\end{align}%
where $\nabla _{S}^{\underline{\mathbf{D}}}(L)=diag\left(
(1-L)^{D_{1}},\cdots ,(1-L)^{D_{s}},\cdots ,(1-L)^{D_{S}}\right) $, $\mathbf{%
\Phi }\left( L\right) =\mathbf{\Phi }_{0}-\mathbf{\Phi }_{1}L-...-\mathbf{%
\Phi }_{P}L^{P}$ and $\left[ \mathbf{\Phi }\left( L\right) \right] ^{-1}=%
\mathbf{\Pi }(L)=\underset{j=0}{\overset{\infty }{\tsum }}\mathbf{\Pi }%
_{j}L^{j}$, with $\left( \mathbf{\Pi }_{j}\right) _{j\in 
\mathbb{N}
}$ is sequence of absolutely summable matrix i.e. $\underset{j=0}{\overset{%
\infty }{\tsum }}\left\vert \mathbf{\Pi }_{j}(l,k)\right\vert <\infty $, $%
\forall l\in \left\{ 1,...,S\right\} $ and $\forall k\in \left\{
1,...,S\right\} $. $C_{j}=\underset{k=0}{\overset{j}{\tsum }}\mathbf{\Psi }%
_{k}\mathbf{\Pi }_{j-k}$ with $\mathbf{\Psi }_{j}$ defined like above. The $%
ith$ element of $\underline{\mathbf{Y}}_{\tau }$, $\mathbf{Y}_{i,\tau }$ is
written as follows%
\begin{equation}
\mathbf{Y}_{i,\tau }=(1-L^{S})^{D_{i}}\left( \mathbf{\Phi }\left( L\right)
^{-1}\right) _{i}\underline{\mathbf{u}}_{\tau }  \tag{$2.4$}
\end{equation}%
where $\left( \mathbf{\Phi }\left( L\right) ^{-1}\right) _{i}$ is the $ith$
rows of $\mathbf{\Phi }\left( L\right) ^{-1}$. From $(2.4)$ we see clearly
that,%
\begin{equation*}
\mathbf{Y}_{i,\tau }\text{ is integrated of order }D_{i},i=1,...,S.
\end{equation*}

\subsection{Periodic autoregressive, S-periodical seasonally fractionally
integrated process ($PAR(p)-PSFI(D_{t})$)}

A periodically correlated process $\left\{ Z_{t},t\in 
\mathbb{Z}
\right\} $ is said, periodic autoregressive of order $p$; $S$-periodical
seasonally fractionally integrated of order $D_{t}$ if it has the following
representation:%
\begin{equation}
\Phi _{t}(L)Z_{t}=(1-L^{S})^{-D_{t}}u_{t},\text{ \ \ \ }t\in 
\mathbb{Z}
,  \tag{$2.5$}
\end{equation}%
where $\left\{ u_{t},t\in 
\mathbb{Z}
\right\} $, $(1-L^{S})^{D_{t}}$ and $\phi _{t}(L)$ are defined like above.
Letting $\underline{\mathbf{Z}}_{\tau }=\left( \mathbf{Z}_{1,\tau },...,%
\mathbf{Z}_{s,\tau },...,\mathbf{Z}_{S,\tau }\right) ^{\prime }$ and $%
\underline{\mathbf{u}}_{\tau }=\left( \mathbf{u}_{1,\tau },...,\mathbf{u}%
_{s,\tau },...,\mathbf{u}_{S,\tau }\right) ^{\prime }$ with $\mathbf{Z}%
_{s,\tau }=Z_{s+S\tau }$ and $\mathbf{u}_{s,\tau }=u_{s+S\tau }$ then the
process $(2.5)$ can be rewritten in the $S$ variate form%
\begin{equation}
\mathbf{\Phi }_{0}\underline{\mathbf{Z}}_{\tau }-\underset{i=1}{\overset{%
\mathbf{P}}{\sum }}\mathbf{\Phi }_{i}\underline{\mathbf{Z}}_{\tau -i}=\nabla
_{S}^{-\underline{\mathbf{D}}}(L)\underline{\mathbf{u}}_{\tau },  \tag{$2.6$}
\end{equation}%
where $\mathbf{P}=\left[ \frac{p+1}{S}\right] +1$, with $\left[ x\right] $
denotes the smallest integer than or equal to $x$, $\nabla _{S}^{-\underline{%
\mathbf{D}}}(L)$, $\mathbf{\Phi }_{0}$ and $\mathbf{\Phi }_{i}$, $i=1,...,%
\mathbf{P}$ are defined like above. The model $(2.6)$ is vector
autoregression with fractional integrated innovation, namely $VARFI$
(Rebecca Sela and Clifford Hurvich $(2008)$). The periodic stationarity
condition of the model $(2.6)$ is the same than the model $(2.2)$. The $ith$
relation of $(2.6$) is written%
\begin{equation*}
\left( \mathbf{\Phi }(L)\right) _{i}\underline{\mathbf{Z}}_{\tau
}=(1-L^{S})^{-D_{i}}\mathbf{u}_{i,\tau }
\end{equation*}%
where $\mathbf{\Phi }\left( L\right) =\mathbf{\Phi }_{0}-\mathbf{\Phi }%
_{1}L-...-\mathbf{\Phi }_{\mathbf{P}}L^{\mathbf{P}}$ and $\left( \mathbf{%
\Phi }(L)\right) _{i}$ is the $ith$ rows of $\mathbf{\Phi }\left( L\right) $%
, this means that the $ith$ relation of $(2.6)$ is integrated of order $%
D_{i} $. Among the $S$ relations of $(2.6)$, those which are integrated of
order lower than $\underset{1\leq i\leq S}{max}D_{i}$ are relations of
cointegration. If all the values $D_{i}$ are different, then we can say that
there are ($S-1$) relations of cointegrations. If $D_{1}=....=D_{S}$ it does
not exist any relation of cointegration. Generally, when we have $%
D_{1}<D_{2}<...D_{S-R-1}<(D_{S-R}=D_{S-R+1}=...=D_{S})$ that means that
there are $\left( S-R-1\right) $ relations of cointegrations between the $S$
seasons. If the model $(2.6)$ is stationary, it has an infinite moving
average representation given by%
\begin{align}
\underline{\mathbf{Z}}_{\tau }& =\Phi \left( L\right) ^{-1}\nabla _{S}^{-%
\underline{\mathbf{D}}}(L)\underline{\mathbf{u}}_{\tau }  \notag \\
& =\left( \underset{j=0}{\overset{\infty }{\tsum }}\mathbf{\Pi }%
_{j}L^{j}\right) \left( \underset{j=0}{\overset{\infty }{\tsum }}\mathbf{%
\Psi }_{j}L^{j}\right) \underline{\mathbf{u}}_{\tau }  \notag \\
& =\underset{j=0}{\overset{\infty }{\tsum }}\left( \underset{k=0}{\overset{j}%
{\tsum }}\mathbf{\Pi }_{k}\mathbf{\Psi }_{j-k}\right) \underline{\mathbf{u}}%
_{\tau }  \notag \\
& =\underset{j=0}{\overset{\infty }{\tsum }}\mathbf{H}_{j}\underline{\mathbf{%
u}}_{\tau -j},  \tag{$2.7$}
\end{align}%
where $\mathbf{H}_{j}=\underset{k=0}{\overset{j}{\tsum }}\mathbf{\Pi }_{k}%
\mathbf{\Psi }_{j-k}$. The $ith$ element of $\underline{\mathbf{Z}}_{\tau }$%
, $\mathbf{Z}_{i,\tau }$ is written as follows%
\begin{equation*}
\mathbf{Z}_{i,\tau }=\overset{S}{\underset{s=1}{\sum }}\left( \mathbf{\Phi }%
(L)^{-1}\right) _{i,s}(1-L)^{-D_{S}}\mathbf{u}_{s,\tau },
\end{equation*}%
where $\left( \mathbf{\Phi }(L)^{-1}\right) _{i,s}$ is $(i,s)$ $th$ element
of the matrix $\mathbf{\Phi }(L)^{-1}$. $\mathbf{Z}_{i,\tau }$ is written
like linear combination of $S$ independent processes, respectively,
integrated of order $D_{1},...,D_{s},...,D_{S}$; consequently $\mathbf{Z}%
_{i,\tau }$ is integrated of order $\underset{1\leq i\leq S}{max}D_{i}$
(Granger $1986$).

\section{Periodic autocovariances}

This section deals with the determination of theoretical periodic
autocovariances of periodically correlated processes defined in precedent
section.

\subsection{$PSFI(D_{t})-PAR(p)$ periodic autocovariances}

\begin{theorem}
Given the stationary S-variate process $\underline{\mathbf{Y}}_{\tau }$
defined by $(2.2)$, we have 
\begin{equation}
\Gamma _{\underline{\mathbf{Y}}_{\tau }}(h)\sim \Delta \left[ h^{D-0.5}%
\right] \mathbf{A}\Delta \left[ h^{D-0.5}\right] \text{, as }h\rightarrow
\infty  \tag{$3.1$}
\end{equation}%
where the $\left( i,k\right) $ $th$ element of $S\times S$ matrix $\mathbf{A}
$ is: 
\begin{equation*}
\frac{\Gamma \left( 1-D_{i}-D_{k}\right) }{\Gamma \left( D_{k}\right) \Gamma
\left( 1-D_{k}\right) }\mathbf{\Pi }_{i}^{\prime }\mathbf{\Omega \Pi }_{k}
\end{equation*}%
with $\mathbf{\Pi }_{i}^{\prime }$ is the $ith$ rows of the matrix $\mathbf{%
\Pi }$ and $\mathbf{\Omega }=diag(\sigma _{1}^{2},...,\sigma
_{s}^{2},...,\sigma _{S}^{2})$.

\begin{proof}
See Ching-Fan Chung (2002).
\end{proof}
\end{theorem}

\begin{corollary}
Given the process $Y_{t}$ defined in $(2.1)$, we have:%
\begin{equation}
\gamma ^{(s)}(j)\sim (h+\delta )^{D_{s}+D_{s+\nu -S\ast \delta }-1}\left[ 
\frac{\Gamma (1-D_{s}-D_{s+\nu -S\ast \delta })}{\Gamma (D_{s+\nu -S\ast
\delta })\Gamma (1-D_{s+\nu -S\ast \delta })}\right] \mathbf{\Pi }%
_{s}^{\prime }\mathbf{\Omega \Pi }_{s+\nu -S\ast \delta },  \tag{$3.2$}
\end{equation}%
where $h$ and $\nu $ are integers such as $j=h\times S+\nu $, and $j>0$,
i.e. $j\equiv \nu \left[ h\right] $ with $0\leq \nu <S-1$ and $\delta $ is
defined as follows:%
\begin{equation*}
\left\{ 
\begin{array}{c}
\delta =0,\text{ \ \ \ \ \ \ \ \ \ \ \ \ \ \ \ \ if }1\leq s+\nu <S, \\ 
\delta =1\text{, \ \ \ \ \ if }S+1\leq s+\nu \leq 2S-1,%
\end{array}%
\right.
\end{equation*}%
$\mathbf{\Pi }_{s}^{\prime }$ is the $sth$ rows of the matrix $\mathbf{\Pi }$
and $\mathbf{\Omega }=diag(\sigma _{1}^{2},...,\sigma _{s}^{2},...,\sigma
_{S}^{2})$.
\end{corollary}

\begin{proof}
The proof of the corollary, rises directly from theorem $1$. From theorem $1$%
, we have :%
\begin{equation}
\Gamma _{\underline{\mathbf{Y}}_{\tau }}^{(i,k)}(h)\sim h^{(D_{i}+D_{k}-1)}%
\frac{\Gamma (1-D_{i}-D_{k})}{\Gamma (D_{k})\Gamma (1-D_{k})}\mathbf{\Pi }%
_{i}^{\prime }\mathbf{\Omega \Pi }_{k},  \tag{$3.3$}
\end{equation}

where $\Gamma _{\underline{Y}_{\tau }}^{(i,k)}(h)=Cov(\mathbf{Y}_{i,\tau },%
\mathbf{Y}_{k,\tau +h})$ are the $(i,k)$ $th$ element of the covariance
matrix of $\Gamma _{\underline{\mathbf{Y}}_{\tau }}(h)$. Moreover, it is
known that%
\begin{align}
\gamma ^{(s)}(j)& =Cov(Y_{S\tau +s},\text{ }Y_{S\tau +s+j})  \notag \\
& =Cov(\mathbf{Y}_{s,\tau },\mathbf{Y}_{s+j,\tau }).  \tag{$3.4$}
\end{align}%
Putting $j=Sh+\nu $ with $0\leq \nu <S-1$, by replacing $j$ by $Sh+\nu $ in $%
(3.4)$, we have%
\begin{equation}
\gamma ^{(s)}(j)=Cov(\mathbf{Y}_{s,\tau },\mathbf{Y}_{s+\nu ,\tau +h}). 
\tag{$3.5$}
\end{equation}%
According to the value of $(s+\nu )$, the equality $(3.5)$, becomes%
\begin{equation*}
\gamma ^{(s)}(j)=\left\{ 
\begin{array}{c}
\Gamma _{\underline{\mathbf{Y}}_{\tau }}^{(s,s+\nu )}(h),\text{ \ \ \ \ \ \
\ \ \ \ \ \ \ \ \ \ \ \ \ \ \ \ \ \ if }1\leq s+\nu <S, \\ 
\Gamma _{\underline{\mathbf{Y}}_{\tau }}^{(s,s+\nu -S)}(h+1),\text{ \ \ if }%
(S+1)\leq s+\nu \leq 2S-1.%
\end{array}%
\right.
\end{equation*}%
By using the approximation $(3.3)$, we have:%
\begin{equation*}
\gamma ^{(s)}(j)\sim \left\{ 
\begin{array}{c}
h^{(D_{s}+D_{s+\nu }-1)}\frac{\Gamma (1-D_{s}-D_{s+\nu })}{\Gamma (D_{s+\nu
})\Gamma (1-D_{s+\nu })}\mathbf{\Pi }_{s}^{\prime }\mathbf{\Omega \Pi }%
_{s+\nu },\text{ \ \ } \\ 
\text{if }1\leq s+\nu <S, \\ 
\\ 
\left( h+1\right) ^{(D_{s}+D_{s+\nu -S\ast \delta }-1)}\frac{\Gamma
(1-D_{s}-D_{s+\nu -S\ast \nu })}{\Gamma (D_{s+\nu })\Gamma (1-D_{s+\nu
-S\ast \delta })}\mathbf{\Pi }_{s}^{\prime }\mathbf{\Omega \Pi }_{s+\nu
-S\ast \delta },\text{ \ \ } \\ 
\text{if }(S+1)\leq s+\nu \leq 2S-1,%
\end{array}%
\right.
\end{equation*}%
where $\mathbf{\Pi }_{s}^{\prime }$ is the $sth$ rows of the matrix $\mathbf{%
\Pi }$ and $\mathbf{\Omega }=diag(\sigma _{1}^{2},...,\sigma
_{s}^{2},...,\sigma _{S}^{2})$. From corollary $1$; emerges several remarks,
the most important are
\end{proof}

\begin{remark}
: The periodic autocovariances $\gamma ^{(s)}(j)$ $s=1,...,S$ taper off at
different hyperbolic rates. If we suppose that $\underset{1\leq i\leq S}{%
\min }D_{i}=D_{1}$ and $\underset{1\leq i\leq S}{\max }D_{i}=D_{S}$ (this
does not restrict the generality) than $\gamma ^{(1)}(j)$, with $j\equiv
0[S] $ has the more speedy taper off hyperbolic rate ($\propto h^{2D_{1}-1}$%
) and $\gamma ^{(S)}(j)$, with $j\equiv 0[S]$ has the lowest taper off
hyperbolic rate ($\propto h^{2D_{S}-1}$).
\end{remark}

This remark will be largely clarified graphically (see section $4$, couples
of figures $(1a,1b)$ to $(5a,5b)$. The advantage which offer by the periodic
process is the possibility of representing the graph of the autocovariances
in various manners. The autocovariances functions $\gamma ^{(s)}(j),$ $%
s=1,...,S$ can be represented in the same plot (hui ad Li $1995$), or
separately. For $j=Sh+\nu $, with $0\leq \nu <S-1$ we can also represented $%
\gamma ^{(s)}(Sh+\nu ),$ $\nu =0,...,S-1$ in the same plot. These are the
three kinds of graphs which we will use in the next section.

\subsection{$PAR(p)-PSFI(D_{t})$ periodic autocovariances}

Before stating the main result of this section, we need some further
notation. Let $D_{\max }=\underset{1\leq s\leq S}{\max }D_{s}$ and define
Before stating the main result of this section, we need some further
notation. Let $D_{\max }=\underset{1\leq s\leq S}{\max }D_{s}$ and define $%
F=\left\{ 1,...,s,...,S\right\} $, $F_{1}=\left\{ i,\text{ }i\in F\text{ / }%
D_{i}=D_{\max }\right\} $, with $\left\vert F_{1}\right\vert =R$ and $%
F_{2}=\left\{ i,\text{ }i\in F\text{ / }D_{i}<D_{\max }\right\} $, with $%
\left\vert F_{2}\right\vert =S-R$. We have $F_{1}\cap F_{2}=\varnothing $
and $F_{1}\cup F_{2}=F$.

\begin{theorem}
Given the stationary S-variate process $\underline{\mathbf{Z}}_{\tau }$%
\textit{\ defined by }$(2.6)$\textit{, we have}%
\begin{equation}
\Gamma _{\underline{\mathbf{Z}}_{\tau }}(j)\sim j^{2D_{\max }-1}\mathbf{A}%
\text{, as }j\rightarrow \infty  \tag{$3.6$}
\end{equation}%
where the $(i,k)$\textit{\ th element of }$S\times S$ matrix $\mathbf{A}$%
\textit{, i}s:%
\begin{equation*}
\mathbf{A}(l,m)=\frac{\Gamma \left( 1-2D_{\max }\right) }{\Gamma \left(
D_{\max }\right) \Gamma \left( 1-D_{\max }\right) }\sum_{i\in F_{1}}\mathbf{%
\Pi }(l,i)\mathbf{\Pi }(m,i)\sigma _{i}^{2}
\end{equation*}%
where $\mathbf{\Pi }=\left[ \mathbf{\Phi }\left( 1\right) \right] ^{-1}=%
\underset{j=0}{\overset{\infty }{\tsum }}\mathbf{\Pi }_{j}$ and $\mathbf{\Pi 
}(l,i)$ is $(l,i)$ th element of $\mathbf{\Pi }$.
\end{theorem}

\begin{proof}
See Ching-Fan Chung $(2002)$.
\end{proof}

The corollary below, gives the approximated expression, as $j\rightarrow
\infty $, of the periodic autocovariances function, $\gamma
^{(s)}(j)=cov(Z_{S\tau +s},$ $Z_{S\tau +s+j})$ of the process $Z_{t}$,
defined in $(2.5)$.

\begin{corollary}
Given the process $Z_{t}$ defined in $(2.5)$, we have,%
\begin{equation}
\gamma ^{(s)}(j)\sim (h+\delta )^{2D_{\max }-1}\left[ \frac{\Gamma
(1-2D_{\max })}{\Gamma (D_{\max })\Gamma (1-D_{\max })}\right] \sum_{i\in
F_{1}}\mathbf{\Pi }(s,i)\mathbf{\Pi }(s+\nu -S\delta ,i)\sigma _{i}^{2} 
\tag{$3.7$}
\end{equation}%
where $h$ and $\nu $ are integers such as $j=hS+\nu $, and $j>0$, i.e. $%
j\equiv \nu \left[ h\right] $ with $0\leq \nu <S-1$, and $\delta $ is
defined as follows:%
\begin{equation*}
\left\{ 
\begin{array}{c}
\delta =0,\text{ \ \ \ \ \ \ \ \ \ \ \ \ \ \ \ \ if }1\leq s+\nu <S \\ 
\delta =1\text{, \ \ \ \ \ if }S+1\leq s+\nu \leq 2S-1%
\end{array}%
\right.
\end{equation*}%
and $\mathbf{\Pi }\left( i,s\right) $ is the $(i,s)$ th element of the
matrix $\mathbf{\Pi }=\left[ \mathbf{\Phi }\left( 1\right) \right] ^{-1}=%
\underset{j=0}{\overset{\infty }{\tsum }}\mathbf{\Pi }_{j}.$
\end{corollary}

\begin{proof}
The proof of the corollary, rises directly from theorem $3$. From theorem $3$%
, we have :%
\begin{equation}
\Gamma _{\underline{\mathbf{Z}}_{\tau }}^{(l,k)}(j)\sim j^{2D_{\max }-1}%
\frac{\Gamma \left( 1-2D_{\max }\right) }{\Gamma \left( D_{\max }\right)
\Gamma \left( 1-D_{\max }\right) }\sum_{i\in F_{1}}\mathbf{\Pi }(l,i)\mathbf{%
\Pi }(k,i)\sigma _{i}^{2}  \tag{$3.8$}
\end{equation}%
where $\Gamma _{\underline{\mathbf{Z}}_{\tau }}^{(l,k)}(j)=Cov(\mathbf{Z}%
_{l,\tau },\mathbf{Z}_{k,\tau +j})$ is the $(i,k)$ $th$ element of the
covariance matrix $\Gamma _{\underline{\mathbf{Z}}_{\tau }}(j)$. Moreover,
it is known that%
\begin{align}
\gamma ^{(s)}(j)& =Cov(Z_{S\tau +s},Z_{S\tau +s+j})  \notag \\
& =Cov(\mathbf{Z}_{s,\tau },\mathbf{Z}_{s+j,\tau })  \tag{$3.9$}
\end{align}%
Putting $j=Sh+\nu $ with $0\leq \nu <S-1$, by replacing $j$ by $Sh+\nu $ in $%
(3.9)$, we have%
\begin{equation}
\gamma ^{(s)}(j)=Cov(\mathbf{Z}_{s,\tau },\mathbf{Z}_{s+\nu ,\tau +h}) 
\tag{$3.10$}
\end{equation}%
According to the value of $(s+\nu )$, the equality $(3.10)$, becomes%
\begin{equation*}
\gamma ^{(s)}(j)=\left\{ 
\begin{array}{c}
\Gamma _{\underline{\mathbf{Z}}_{\tau }}^{(s,s+\nu )}(h),\text{ \ \ \ \ \ \
\ \ \ \ \ \ \ \ \ \ \ \ \ \ \ \ \ \ if }1\leq s+\nu <S \\ 
\Gamma _{\underline{\mathbf{Z}}_{\tau }}^{(s,s+\nu -S)}(h+1),\text{ \ \ if }%
(S+1)\leq s+\nu \leq 2S-1%
\end{array}%
\right.
\end{equation*}%
By using the approximation $(3.8)$, we have,%
\begin{equation*}
\gamma ^{(s)}(j)\sim \left\{ 
\begin{array}{c}
h^{(2D_{\max }-1)}\frac{\Gamma \left( 1-2D_{\max }\right) }{\Gamma \left(
D_{\max }\right) \Gamma \left( 1-D_{\max }\right) }\sum_{i\in F_{1}}\mathbf{%
\Pi }(s,i)\mathbf{\Pi }(s+\nu ,i)\sigma _{i}^{2},\text{ \ \ \ \ \ \ \ \ \ \
\ \ \ \ } \\ 
\text{if }1\leq s+\nu <S \\ 
\\ 
\\ 
(h+1)^{(2D_{\max }-1)}\frac{\Gamma \left( 1-2D_{\max }\right) }{\Gamma
\left( D_{\max }\right) \Gamma \left( 1-D_{\max }\right) }\sum_{i\in F_{1}}%
\mathbf{\Pi }(s,i)\mathbf{\Pi }(s+\nu -S,i)\sigma _{i}^{2}, \\ 
\text{if }(S+1)\leq s+\nu \leq 2S-1%
\end{array}%
\right.
\end{equation*}
\end{proof}

\begin{remark}
If $D_{1}=D_{2}=...=D_{S}$ the periodic autocovariances $\gamma ^{(s)}(j)$
of the model $(2.2)$ coincide with those of model $(2.8)$
\end{remark}

\begin{remark}
From corollary $4$; we see that the periodic autocovariances $\gamma
^{(s)}(j)$ $s=1,...,S$ taper off at the same hyperbolic rates.
\end{remark}

\section{Simulation}

In this section we compare the finite sample of the periodic autocovariances 
$\gamma ^{(s)}(j)$ $s=1,...,4$ of the models $(1.3)$, $(2.1)$ and $(2.5)$
for different value of $\underline{D}=(D_{1},D_{2},D_{3},D_{4})$. The sample
size for each model is $T=1000.$

The model we consider for the simulation study are

\begin{itemize}
\item \textbf{Model }$\mathbf{A}$%
\begin{equation*}
(1-L^{4})^{D_{t}}X_{t}=\varepsilon _{t}
\end{equation*}%
which has the following S-variate representation%
\begin{equation*}
\left( 
\begin{array}{cccc}
(1-L)^{D_{1}} & 0 & 0 & 0 \\ 
0 & (1-L)^{D_{2}} & 0 & 0 \\ 
0 & 0 & (1-L)^{D_{3}} & 0 \\ 
0 & 0 & 0 & (1-L)^{D_{4}}%
\end{array}%
\right) \underline{\mathbf{X}}_{\tau }=\underline{\mathbf{u}}_{\tau }
\end{equation*}

\item \textbf{Model }$\mathbf{B}$%
\begin{equation*}
\Phi _{t}(L)(1-L^{4})^{D_{t}}Y_{t}=u_{t}
\end{equation*}%
which has the following S-variate representation%
\begin{equation*}
\left( 
\begin{array}{cccc}
1 & 0 & 0 & -0.7 \\ 
-0.8 & 1 & 0 & 0 \\ 
0 & -0.6 & 1 & 0 \\ 
0 & 0 & -0.4 & 1%
\end{array}%
\right) \left( 
\begin{array}{cccc}
(1-L)^{D_{1}} & 0 & 0 & 0 \\ 
0 & (1-L)^{D_{2}} & 0 & 0 \\ 
0 & 0 & (1-L)^{D_{3}} & 0 \\ 
0 & 0 & 0 & (1-L)^{D_{4}}%
\end{array}%
\right) \underline{\mathbf{Y}}_{\tau }=\underline{\mathbf{u}}_{\tau }
\end{equation*}

\item \textbf{Model }$\mathbf{C}$%
\begin{equation*}
(1-L^{4})^{D_{t}}\Phi _{t}(L)Z_{t}=u_{t}
\end{equation*}%
which has the following S-variate representation%
\begin{equation*}
\left( 
\begin{array}{cccc}
(1-L)^{D_{1}} & 0 & 0 & 0 \\ 
0 & (1-L)^{D_{2}} & 0 & 0 \\ 
0 & 0 & (1-L)^{D_{3}} & 0 \\ 
0 & 0 & 0 & (1-L)^{D_{4}}%
\end{array}%
\right) \left( 
\begin{array}{cccc}
1 & 0 & 0 & -0.7 \\ 
-0.8 & 1 & 0 & 0 \\ 
0 & -0.6 & 1 & 0 \\ 
0 & 0 & -0.4 & 1%
\end{array}%
\right) \underline{\mathbf{Z}}_{\tau }=\underline{\mathbf{u}}_{\tau }
\end{equation*}%
where $\underline{\mathbf{u}}_{\tau }$ are $i.i.d$ $N(\mathbf{0},\mathbf{%
\Omega })$ with $\mathbf{\Omega }=diag(1,1,1,1)$.
\end{itemize}

\subsection{Simulated Autocovariances of model A}

In figures $1$ to $4$, we represent the empirical autocovariances function $%
\gamma ^{(s)}(j)$ $s=1,...,4$ in the same plot, for the model $A$ for
different value of $\underline{D}=(D_{1},D_{2},D_{3},D_{4})$.%
\begin{equation*}
\begin{array}{c}
\FRAME{itbpF}{2.1932in}{1.6924in}{0in}{}{}{Figure}{\special{language
"Scientific Word";type "GRAPHIC";maintain-aspect-ratio TRUE;display
"USEDEF";valid_file "T";width 2.1932in;height 1.6924in;depth
0in;original-width 10.8283in;original-height 8.3307in;cropleft "0";croptop
"1";cropright "1";cropbottom "0";tempfilename
'O3YZLR00.wmf';tempfile-properties "XPR";}} \\ 
\text{\textbf{Figure 1}: The periodic autocovariances }\widehat{\gamma }%
^{(s)}(j)\text{, }s=1,...,4\text{, for lag }j=1\text{ to }25 \\ 
\text{with }\underline{D}=(0.1,\text{ }0.2,\text{ }0.3,\text{ }0.4)\text{
for model A} \\ 
\FRAME{itbpF}{2.1932in}{1.6924in}{0in}{}{}{Figure}{\special{language
"Scientific Word";type "GRAPHIC";maintain-aspect-ratio TRUE;display
"USEDEF";valid_file "T";width 2.1932in;height 1.6924in;depth
0in;original-width 10.8283in;original-height 8.3307in;cropleft "0";croptop
"1";cropright "1";cropbottom "0";tempfilename
'O3YZLR01.wmf';tempfile-properties "XPR";}} \\ 
\text{\textbf{Figure 2}: The periodic autocovariances }\widehat{\gamma }%
^{(s)}(j)\text{, }s=1,...,4\text{, for lag }j=1\text{ to }25 \\ 
\text{with }\underline{D}=(0.1,\text{ }0.2,\text{ }0.4,\text{ }0.4)\text{
for model A} \\ 
\FRAME{itbpF}{2.1932in}{1.6924in}{0in}{}{}{Figure}{\special{language
"Scientific Word";type "GRAPHIC";maintain-aspect-ratio TRUE;display
"USEDEF";valid_file "T";width 2.1932in;height 1.6924in;depth
0in;original-width 10.8283in;original-height 8.3307in;cropleft "0";croptop
"1";cropright "1";cropbottom "0";tempfilename
'O3YZLR02.wmf';tempfile-properties "XPR";}} \\ 
\text{\textbf{Figure 3}: The periodic autocovariances }\widehat{\gamma }%
^{(s)}(j)\text{, }s=1,...,4\text{, for lag }j=1\text{ to }25 \\ 
\text{with }\underline{D}=(0.1,\text{ }0.4,\text{ }0.4,\text{ }0.4)\text{
for model A} \\ 
\FRAME{itbpF}{2.1932in}{1.6924in}{0in}{}{}{Figure}{\special{language
"Scientific Word";type "GRAPHIC";maintain-aspect-ratio TRUE;display
"USEDEF";valid_file "T";width 2.1932in;height 1.6924in;depth
0in;original-width 10.8283in;original-height 8.3307in;cropleft "0";croptop
"1";cropright "1";cropbottom "0";tempfilename
'O3YZLR03.wmf';tempfile-properties "XPR";}} \\ 
\text{\textbf{Figure 4}: The periodic autocovariances }\widehat{\gamma }%
^{(s)}(j)\text{, }s=1,...,4\text{, for lag }j=1\text{ to }25 \\ 
\text{with }\underline{D}=(0.4,\text{ }0.4,\text{ }0.4,\text{ }0.4)\text{
for model A}%
\end{array}%
\end{equation*}%
The figures $1$, illustrate well the theoretical result of theorem $1$ and
also states that the periodicity is caused by the fractional parameters $%
\underline{D}=(0.1,0.2,0.3,0.4)$ (the auto-covariances $\gamma ^{(s)}(j)$ $%
s=1,...,4$ for lag $j\equiv 0[4]$ taper off, respectively, at hyperbolic
rates, according the value of $\underline{D}$.

\subsection{Simulated autocovariances of model B}

For $\underline{D}=(0.1,0.2,0.3,0.4)$, the figures $\left( 1a\right) $ and $%
\left( 1b\right) $ represents the empirical autocovariances $\widehat{\gamma 
}^{(s)}(j)$ $s=1,...,4$, respectively, in spike graph and in line graph of
the model $B$. The couples of figures $(2a,2b)$ to $(5a,5b)$ represents the
empirical autocovariances $\widehat{\gamma }^{(1)}(4h+\nu ),$ $\nu =0,...,3$
to $\widehat{\gamma }^{(4)}(4h+\nu ),$ $\nu =0,...,3$, for $h=1$ to $25,$
respectively, in spike graph and line graph, for the model $B$.

\begin{eqnarray*}
&&%
\begin{array}{cc}
\FRAME{itbpF}{2.1534in}{1.6942in}{0in}{}{}{Figure}{\special{language
"Scientific Word";type "GRAPHIC";maintain-aspect-ratio TRUE;display
"USEDEF";valid_file "T";width 2.1534in;height 1.6942in;depth
0in;original-width 10.6285in;original-height 8.3307in;cropleft "0";croptop
"1";cropright "1";cropbottom "0";tempfilename
'O3YZLR04.wmf';tempfile-properties "XPR";}} & \FRAME{itbpF}{2.1534in}{%
1.6942in}{0in}{}{}{Figure}{\special{language "Scientific Word";type
"GRAPHIC";maintain-aspect-ratio TRUE;display "USEDEF";valid_file "T";width
2.1534in;height 1.6942in;depth 0in;original-width 10.6285in;original-height
8.3307in;cropleft "0";croptop "1";cropright "1";cropbottom "0";tempfilename
'O3YZLR05.wmf';tempfile-properties "XPR";}} \\ 
\text{Figure }1a & \text{Figure }1b%
\end{array}
\\
&&%
\begin{array}{c}
\text{The periodic autocovariances }\widehat{\gamma }^{(s)}(j)\text{, }%
s=1,...,4\text{, for lag }j=1\text{ to }100\text{,} \\ 
\text{of model B, with }\underline{D}=(0.1,\text{ }0.2,\text{ }0.3,\text{ }%
0.4),\text{ taper off at different hyperbolic rates }%
\end{array}%
\end{eqnarray*}

\begin{eqnarray*}
&&%
\begin{array}{cc}
\begin{array}{c}
\FRAME{itbpF}{2.1932in}{1.6933in}{0in}{}{}{Figure}{\special{language
"Scientific Word";type "GRAPHIC";maintain-aspect-ratio TRUE;display
"USEDEF";valid_file "T";width 2.1932in;height 1.6933in;depth
0in;original-width 10.8283in;original-height 8.3307in;cropleft "0";croptop
"1";cropright "1";cropbottom "0";tempfilename
'O3YZLR06.wmf';tempfile-properties "XPR";}} \\ 
\text{figure }2a%
\end{array}
& 
\begin{array}{c}
\FRAME{itbpF}{2.1534in}{1.6942in}{0in}{}{}{Figure}{\special{language
"Scientific Word";type "GRAPHIC";maintain-aspect-ratio TRUE;display
"USEDEF";valid_file "T";width 2.1534in;height 1.6942in;depth
0in;original-width 10.6285in;original-height 8.3307in;cropleft "0";croptop
"1";cropright "1";cropbottom "0";tempfilename
'O3YZLR07.wmf';tempfile-properties "XPR";}} \\ 
\text{figure }2b%
\end{array}%
\end{array}
\\
&&%
\begin{array}{c}
\text{The figure }2a\text{\ and }2b\text{\ represents, respectively, the
speedy and the lowest } \\ 
\text{taper off hyperbolic rate of autocovariances of model (}B\text{)}%
\end{array}%
\end{eqnarray*}%
\begin{equation*}
\begin{array}{cc}
\FRAME{itbpF}{2.1932in}{1.6933in}{0in}{}{}{Figure}{\special{language
"Scientific Word";type "GRAPHIC";maintain-aspect-ratio TRUE;display
"USEDEF";valid_file "T";width 2.1932in;height 1.6933in;depth
0in;original-width 10.8283in;original-height 8.3307in;cropleft "0";croptop
"1";cropright "1";cropbottom "0";tempfilename
'O3YZLR08.wmf';tempfile-properties "XPR";}} & \FRAME{itbpF}{2.1932in}{%
1.6933in}{0in}{}{}{Figure}{\special{language "Scientific Word";type
"GRAPHIC";maintain-aspect-ratio TRUE;display "USEDEF";valid_file "T";width
2.1932in;height 1.6933in;depth 0in;original-width 10.8283in;original-height
8.3307in;cropleft "0";croptop "1";cropright "1";cropbottom "0";tempfilename
'O3YZLR09.wmf';tempfile-properties "XPR";}} \\ 
\text{Figure }2a & \text{Figure }2b%
\end{array}%
\end{equation*}%
\begin{equation*}
\begin{array}{c}
\text{The periodic autocovariances }\gamma ^{(1)}(4h+\nu ),\text{ }\nu
=0,...,3,\text{ for fixed }h\text{ (}h=1\text{ to }25\text{)} \\ 
\text{have tendency to increase according with the value of }D_{1}+D_{1+\nu
-4\delta }%
\end{array}%
\end{equation*}%
\begin{equation*}
\begin{array}{cc}
\begin{array}{c}
\FRAME{itbpF}{2.1534in}{1.6942in}{0in}{}{}{Figure}{\special{language
"Scientific Word";type "GRAPHIC";maintain-aspect-ratio TRUE;display
"USEDEF";valid_file "T";width 2.1534in;height 1.6942in;depth
0in;original-width 10.6285in;original-height 8.3307in;cropleft "0";croptop
"1";cropright "1";cropbottom "0";tempfilename
'O3YZLR0A.wmf';tempfile-properties "XPR";}} \\ 
\text{Figure }3a%
\end{array}
& 
\begin{array}{c}
\FRAME{itbpF}{2.1534in}{1.6942in}{0in}{}{}{Figure}{\special{language
"Scientific Word";type "GRAPHIC";maintain-aspect-ratio TRUE;display
"USEDEF";valid_file "T";width 2.1534in;height 1.6942in;depth
0in;original-width 10.6285in;original-height 8.3307in;cropleft "0";croptop
"1";cropright "1";cropbottom "0";tempfilename
'O3YZLR0B.wmf';tempfile-properties "XPR";}} \\ 
\text{Figure }3b%
\end{array}%
\end{array}%
\end{equation*}%
\begin{equation*}
\begin{array}{c}
\text{The periodic autocovariances }\widehat{\gamma }^{(2)}(4h+\nu ),\text{ }%
\nu =0,...,3,\text{ for fixed }h\text{ (}h=1\text{ to }25\text{)} \\ 
\text{have tendency to increase according with the value of }D_{2}+D_{2+\nu
-4\delta }%
\end{array}%
\end{equation*}%
\begin{equation*}
\begin{array}{cc}
\begin{array}{c}
\FRAME{itbpF}{2.1534in}{1.6942in}{0in}{}{}{Figure}{\special{language
"Scientific Word";type "GRAPHIC";maintain-aspect-ratio TRUE;display
"USEDEF";valid_file "T";width 2.1534in;height 1.6942in;depth
0in;original-width 10.6285in;original-height 8.3307in;cropleft "0";croptop
"1";cropright "1";cropbottom "0";tempfilename
'O3YZLR0C.wmf';tempfile-properties "XPR";}} \\ 
\text{Figure }4a%
\end{array}
& 
\begin{array}{c}
\FRAME{itbpF}{2.1534in}{1.6942in}{0in}{}{}{Figure}{\special{language
"Scientific Word";type "GRAPHIC";maintain-aspect-ratio TRUE;display
"USEDEF";valid_file "T";width 2.1534in;height 1.6942in;depth
0in;original-width 10.6285in;original-height 8.3307in;cropleft "0";croptop
"1";cropright "1";cropbottom "0";tempfilename
'O3YZLR0D.wmf';tempfile-properties "XPR";}} \\ 
\text{Figure }4b%
\end{array}%
\end{array}%
\end{equation*}%
\begin{equation*}
\begin{array}{c}
\text{The periodic autocovariances }\gamma ^{(3)}(4h+\nu ),\text{ }\nu
=0,...,3,\text{ for fixed }h\text{ (}h=1\text{ to }24\text{)} \\ 
\text{have tendency to increase according with the value of }D_{3}+D_{3+\nu
-4\delta }%
\end{array}%
\end{equation*}%
\begin{equation*}
\begin{array}{cc}
\begin{array}{c}
\FRAME{itbpF}{2.1534in}{1.6942in}{0in}{}{}{Figure}{\special{language
"Scientific Word";type "GRAPHIC";maintain-aspect-ratio TRUE;display
"USEDEF";valid_file "T";width 2.1534in;height 1.6942in;depth
0in;original-width 10.6285in;original-height 8.3307in;cropleft "0";croptop
"1";cropright "1";cropbottom "0";tempfilename
'O3YZLR0E.wmf';tempfile-properties "XPR";}} \\ 
\text{figure 5a}%
\end{array}
& 
\begin{array}{c}
\FRAME{itbpF}{2.1534in}{1.6942in}{0in}{}{}{Figure}{\special{language
"Scientific Word";type "GRAPHIC";maintain-aspect-ratio TRUE;display
"USEDEF";valid_file "T";width 2.1534in;height 1.6942in;depth
0in;original-width 10.6285in;original-height 8.3307in;cropleft "0";croptop
"1";cropright "1";cropbottom "0";tempfilename
'O3YZLR0F.wmf';tempfile-properties "XPR";}} \\ 
\text{figure 5b}%
\end{array}%
\end{array}%
\end{equation*}%
\begin{equation*}
\begin{array}{c}
\text{The periodic autocovariances }\widehat{\gamma }^{(4)}(4h+\nu ),\text{ }%
\nu =0,...,3,\text{ for fixed }h\text{ (}h=1\text{ to }25\text{)} \\ 
\text{have tendency to increase according with the value of }D_{4}+D_{4+\nu
-4\delta }%
\end{array}%
\end{equation*}

\subsection{Simulated autocovariances of model C}

The figures $\left( 1c\right) $ and $\left( 2c\right) $ represents the
empirical autocovariances $\widehat{\gamma }^{(s)}(j)$ $s=1,...,4$,
respectively, in spike graph and in line graph of the model $C$. The
difference between the periodic autocovariances $\gamma ^{(s)}(j)$, $%
s=1,...,4$, for lag $j=1$ to $100$, decreases at the same manner, mainly
because they taper off at the same hyperbolic rates 
\begin{equation*}
\begin{array}{cc}
\FRAME{itbpF}{2.1932in}{1.6933in}{0in}{}{}{Figure}{\special{language
"Scientific Word";type "GRAPHIC";maintain-aspect-ratio TRUE;display
"USEDEF";valid_file "T";width 2.1932in;height 1.6933in;depth
0in;original-width 10.8283in;original-height 8.3307in;cropleft "0";croptop
"1";cropright "1";cropbottom "0";tempfilename
'O3YZLR0G.wmf';tempfile-properties "XPR";}} & \FRAME{itbpF}{2.1932in}{%
1.6933in}{0in}{}{}{Figure}{\special{language "Scientific Word";type
"GRAPHIC";maintain-aspect-ratio TRUE;display "USEDEF";valid_file "T";width
2.1932in;height 1.6933in;depth 0in;original-width 10.8283in;original-height
8.3307in;cropleft "0";croptop "1";cropright "1";cropbottom "0";tempfilename
'O3YZLR0H.wmf';tempfile-properties "XPR";}} \\ 
\text{Figure }1c & \text{Figure }2c%
\end{array}%
\end{equation*}%
\begin{equation*}
\begin{array}{c}
\text{The periodic autocovariances }\widehat{\gamma }^{(s)}(j)\text{, }%
s=1,...,4\text{, for lag }j=1\text{ to }100\text{ and } \\ 
\underline{D}=(0.1,\text{ }0.2,\text{ }0.4,\text{ }0.4),\text{ taper off at
the same hyperbolic rates}%
\end{array}%
\end{equation*}

\subsection{Simulated comparison between autocovariances of model B and C}

In order to compare, both autocovariances $\widehat{\gamma }^{(s)}(j)$, $%
s=1,...,4$ for model $\left( B\right) $ and model $\left( C\right) $ we
represent them graphically in the same scale for different value of $%
\underline{D}=(D_{1},D_{2},D_{3},D_{4})$ (see below).

\begin{equation*}
\begin{array}{c}
\begin{array}{cc}
\FRAME{itbpF}{2.1534in}{1.6942in}{0in}{}{}{Figure}{\special{language
"Scientific Word";type "GRAPHIC";maintain-aspect-ratio TRUE;display
"USEDEF";valid_file "T";width 2.1534in;height 1.6942in;depth
0in;original-width 10.6285in;original-height 8.3307in;cropleft "0";croptop
"1";cropright "1";cropbottom "0";tempfilename
'O3YZLR0K.wmf';tempfile-properties "XPR";}} & \FRAME{itbpF}{2.1534in}{%
1.6942in}{0in}{}{}{Figure}{\special{language "Scientific Word";type
"GRAPHIC";maintain-aspect-ratio TRUE;display "USEDEF";valid_file "T";width
2.1534in;height 1.6942in;depth 0in;original-width 10.6285in;original-height
8.3307in;cropleft "0";croptop "1";cropright "1";cropbottom "0";tempfilename
'O3YZLR0L.wmf';tempfile-properties "XPR";}}%
\end{array}
\\ 
\text{\textbf{Figure 5}: The periodic autocovariances }\widehat{\gamma }%
^{(s)}(j)\text{, }s=1,...,4\text{, for lag }j=1\text{ to }100, \\ 
\underline{D}=(0.1,\text{ }0.2,\text{ }0.3,\text{ }0.4)\text{ for
respectively, model }\left( B\right) \text{ and model }\left( C\right)%
\end{array}%
\end{equation*}%
\begin{equation*}
\begin{array}{c}
\begin{array}{cc}
\FRAME{itbpF}{2.1932in}{1.6933in}{0in}{}{}{Figure}{\special{language
"Scientific Word";type "GRAPHIC";maintain-aspect-ratio TRUE;display
"USEDEF";valid_file "T";width 2.1932in;height 1.6933in;depth
0in;original-width 10.8283in;original-height 8.3307in;cropleft "0";croptop
"1";cropright "1";cropbottom "0";tempfilename
'O3YZLR0M.wmf';tempfile-properties "XPR";}} & \FRAME{itbpF}{2.1932in}{%
1.6933in}{0in}{}{}{Figure}{\special{language "Scientific Word";type
"GRAPHIC";maintain-aspect-ratio TRUE;display "USEDEF";valid_file "T";width
2.1932in;height 1.6933in;depth 0in;original-width 10.8283in;original-height
8.3307in;cropleft "0";croptop "1";cropright "1";cropbottom "0";tempfilename
'O3YZLR0N.wmf';tempfile-properties "XPR";}}%
\end{array}
\\ 
\begin{array}{c}
\text{\textbf{Figure 6}: The periodic autocovariances }\widehat{\gamma }%
^{(s)}(j)\text{, }s=1,...,4\text{, for lag }j=1\text{ to }100, \\ 
\underline{D}=(0.1,\text{ }0.2,\text{ }0.4,\text{ }0.4)\text{ for
respectively, model }\left( B\right) \text{ and model }\left( C\right)%
\end{array}%
\end{array}%
\end{equation*}

\begin{equation*}
\begin{array}{c}
\begin{array}{cc}
\FRAME{itbpF}{2.1534in}{1.6942in}{0in}{}{}{Figure}{\special{language
"Scientific Word";type "GRAPHIC";maintain-aspect-ratio TRUE;display
"USEDEF";valid_file "T";width 2.1534in;height 1.6942in;depth
0in;original-width 10.6285in;original-height 8.3307in;cropleft "0";croptop
"1";cropright "1";cropbottom "0";tempfilename
'O3YZLR0O.wmf';tempfile-properties "XPR";}} & \FRAME{itbpF}{2.1534in}{%
1.6942in}{0in}{}{}{Figure}{\special{language "Scientific Word";type
"GRAPHIC";maintain-aspect-ratio TRUE;display "USEDEF";valid_file "T";width
2.1534in;height 1.6942in;depth 0in;original-width 10.6285in;original-height
8.3307in;cropleft "0";croptop "1";cropright "1";cropbottom "0";tempfilename
'O3YZLR0P.wmf';tempfile-properties "XPR";}}%
\end{array}
\\ 
\text{\textbf{Figure 7}: The periodic autocovariances }\widehat{\gamma }%
^{(s)}(j)\text{, }s=1,...,4\text{, for lag }j=1\text{ to }100, \\ 
\underline{D}=(0.1,\text{ }0.4,\text{ }0.4,\text{ }0.4)\text{ for,
respectively, model }\left( B\right) \text{ and model }\left( C\right)%
\end{array}%
\end{equation*}

\begin{equation*}
\begin{array}{c}
\begin{array}{cc}
\FRAME{itbpF}{2.1932in}{1.6933in}{0in}{}{}{Figure}{\special{language
"Scientific Word";type "GRAPHIC";maintain-aspect-ratio TRUE;display
"USEDEF";valid_file "T";width 2.1932in;height 1.6933in;depth
0in;original-width 10.8283in;original-height 8.3307in;cropleft "0";croptop
"1";cropright "1";cropbottom "0";tempfilename
'O3YZLR0Q.wmf';tempfile-properties "XPR";}} & \FRAME{itbpF}{2.1932in}{%
1.6933in}{0in}{}{}{Figure}{\special{language "Scientific Word";type
"GRAPHIC";maintain-aspect-ratio TRUE;display "USEDEF";valid_file "T";width
2.1932in;height 1.6933in;depth 0in;original-width 10.8283in;original-height
8.3307in;cropleft "0";croptop "1";cropright "1";cropbottom "0";tempfilename
'O3YZLR0R.wmf';tempfile-properties "XPR";}}%
\end{array}
\\ 
\text{\textbf{Figure 8}: The periodic autocovariances }\widehat{\gamma }%
^{(s)}(j)\text{, }s=1,...,4\text{, for lag }j=1\text{ to }100, \\ 
\underline{D}=(0.4,\text{ }0.4,\text{ }0.4,\text{ }0.4)\text{ for,
respectively, model }\left( B\right) \text{ and model }\left( C\right)%
\end{array}%
\end{equation*}

In figures $5$ to $8$ we plot the autocovariance sequences $\widehat{\gamma }%
^{(s)}(j)$, $s=1,...,4$ of model B and model C in the same scale and with
identical parameters ($\Phi (L)$, $\Omega $, and $\underline{D})$. The
autocovariances sequences differ dramatically. Rebecca Sela and Clifford
Hurvich $(2008)$ presents a similar conclusion for cross-covariance
sequences of bivariate $FIVAR(1,\underline{D})$ and $VARFI(1,\underline{D})$
processes with the same parameters. They point out that the first model have
the series integrated separately (in our case the seasons are integrated
separately) and in the second there is cointegration relation between the
two series (in our case there are $3$ cointegrations relations between the
four seasons). This fact, does not explain clearly why there is such
difference between the autocovariances of model $B$ and model $C$. Further
more, the taper off hyperbolic rates of the autocovariances of model $(C)$
is equal than the lowest tapper off hyperbolic rate of the autocovariances
of model $(B)$, so why the autocovariance sequences differ dramatically? The
explanation is in explicit results of corollary $3.1$ and corollary $3.2$.
Generally, in the literature of long memory models, attention is focused on
the fractional parameters (which associate with hyperbolic tapper off of
autocovariance) rather than on autoregressive or moving average parameters
and $V(\varepsilon _{t})$ included in expression of autocovariance. In the
expression $(3.4)$, the autoregressive parameters and $V(\varepsilon _{t})$
appears in the following form: $\pi _{s}^{\prime }\Omega \pi _{s+\nu -S\ast
\delta }$ and in expression $(3.6)$ it appears in the following form: $\Pi
(s,S)\Pi (s+\nu -S\ast \delta ,S)\sigma _{S}^{2}$. From model $(B)$ and
model $(C)$, the set, of possible values, of these two quantities are
respectively:%
\begin{equation}
\left( 
\begin{array}{cccc}
2.131 & 1.8989 & 1.4628 & 1.3938 \\ 
1.8989 & 2.6744 & 1.8634 & 1.3923 \\ 
1.4628 & 1.8634 & 2.2734 & 1.2975 \\ 
1.3938 & 1.3923 & 1.2975 & 1.6743%
\end{array}%
\right)  \tag{$4.1$}
\end{equation}%
and%
\begin{equation}
\left( 
\begin{array}{cccc}
0.65398 & 0.52318 & 0.31391 & 0.93428 \\ 
0.52318 & 0.41854 & 0.25113 & 0.74742 \\ 
0.31391 & 0.25113 & 0.15068 & 0.44845 \\ 
0.93428 & 0.74742 & 0.4445 & 1.3347%
\end{array}%
\right)  \tag{$4.2$}
\end{equation}%
It seen that all, possible values, of $\pi _{s}^{\prime }\Omega \pi _{s+\nu
-S\ast \delta }$ are greater than $1$ (some are greater than $2$, see the
diagonal of matrix $(4.1)$). On the other hand, all values of $\Pi (s,S)\Pi
(s+\nu -S\ast \delta ,S)\sigma _{S}^{2}$ are lower than $1$ (except the last
value in the diagonal of matrix $(4.2)$.

\section{Conclusion}

For Seasonal-Periodic-$ARFIMA(p,0,0)(0,D,0)$ model, allowing the seasonal
fractional parameter D to be S-periodic rather than constant we have
highlighted the existence of two distinct models (see section 1, model(I)
and model (II)). For these two distinct models we have established the exact
and approximated expression of the periodic autocovariance. On the simulated
sample, for each model, the empirical periodic autocovariance are calculated
and graphically represented.

It is clear, through, theoretical and simulated results that it is not easy
to distinguish between these two models (the shape of the autocovariance for
each model is not sufficient). If we consider the general model, namely,
Seasonal-Periodic-$ARFIMA(p,d_{t},q)(P,D_{t},Q)$ the situation becomes more
complex to handle, because the number of different models we can distinguish
is more than two models. Furthermore, the non seasonal part of the general
model (i.e. $PARFIMA(p,d_{t},q)$) did not receive much attention on behalf
of the statisticians and the probabilists.

\section{\protect\LARGE Appendix A}

\underline{\textbf{Proposition}}\textbf{: }\textit{The infinite moving
average representation of the process, }$\left\{ y_{t},t\in 
\mathbb{Z}
\right\} ,$\textit{\ defined by }$(1.1)$\textit{, is given by}\newline

\noindent $y_{t}=u_{t}+\sum\limits_{j=1}^{\infty }\left(
\sum\limits_{k=1}^{j}(-1)^{k}\sum\limits_{i_{1}+i_{2}+\cdots +i_{k}=j}\pi
_{i_{1}}(-d_{t})\pi _{i_{2}}(-d_{t-i_{1}})\pi
_{i_{3}}(-d_{t-i_{1}-i_{2}})\cdots \pi _{i_{k}}\left(
-d_{t-\tsum\nolimits_{l=1}^{k-1}i_{l}}\right) \right) u_{t-j},$\newline

\textit{where }$\pi _{0}(-d_{t})=1$ and $\pi _{i_{1}}(-d_{t})=\frac{%
-d_{t}\left( -d_{t}+1\right) \cdots \left( -d_{t}+i_{1}-1\right) }{\left(
i_{1}\right) !}.$ The number terms in the sum $"\sum_{i_{1}+i_{2}+\cdots
+i_{k}=j}"$ is equal $2^{j-1}$. The number $2^{j-1}$ represent the cardinal
sets of $k$ positive integers, namely, ($i_{1},i_{2},\cdots ,i_{k}$), which
when summed together give $j$.

\begin{proof}
Putting%
\begin{equation*}
\Pi _{0}(-d_{t})=1\text{ and }\frac{\Gamma \left( j-d_{t}\right) }{\Gamma
\left( j+1\right) \Gamma \left( -d_{t}\right) }=\Pi _{j}(-d_{t}),
\end{equation*}%
we can rewrite $(1.4)$ as%
\begin{equation}
y_{t}+\sum_{j=1}^{\infty }\Pi _{j}(-d_{t})y_{t-j}=u_{t}.  \tag{$A1$}
\end{equation}
\end{proof}

More generally, we have%
\begin{equation*}
y_{t-j}+\sum_{k=1}^{\infty}\Pi_{j}(-d_{t-j})y_{t-j-k}=u_{t-j}.
\end{equation*}

Suppose that the infinite moving average representation of $(A1)$ is given by%
\begin{equation}
y_{t}=\Psi _{0}(t)u_{t}+\sum_{j=1}^{\infty }\Psi _{j}(t)u_{t-j}\text{, with }%
\Psi _{0}=1,  \tag{$A2$}
\end{equation}%
we have for the lagged variable $y_{t-j}$,%
\begin{equation}
y_{t-j}=u_{t-j}+\sum_{k=1}^{\infty }\Psi _{k}(t-j)u_{t-j-k}.  \tag{$A3$}
\end{equation}

By replacing $y_{t-j}$ by $u_{t-j}+\sum_{k=1}^{\infty }\Psi
_{k}(t-j)u_{t-j-k}$ in $\left( A1\right) $, we obtain%
\begin{equation}
Y_{t}+\sum_{j=1}^{\infty }\Pi _{j}(-d_{t})u_{t-j}+\sum_{j=1}^{\infty
}\sum_{k=1}^{\infty }\Pi _{j}(-d_{t})\Psi _{k}(t-j)u_{t-j-k}=u_{t}. 
\tag{$A4$}
\end{equation}

Putting $k^{\prime }=k+j$, $(A4)$ becomes%
\begin{equation}
Y_{t}+\sum_{j=1}^{\infty }\Pi _{j}(-d_{t})u_{t-j}+\sum_{j=1}^{\infty
}\sum_{k^{\prime }=j+1}^{\infty }\Pi _{j}(-d_{t})\Psi _{k^{\prime
}-j}(t-j)u_{t-k^{\prime }}=u_{t}.  \tag{$A5$}
\end{equation}

Let $\Phi _{j,k^{\prime }}(t-j)=\Pi _{j}(-d_{t})\Psi _{k^{\prime }-j}(t-j)$,
then we can rewrite $(A5)$ as,\newpage

\noindent ${\small Y}_{t}{\small +\Pi }_{1}{\small (-d}_{t}{\small )u}%
_{t-1}\tsum\limits_{j=2}^{\infty }{\small \Pi }_{j}{\small (-d}_{t}{\small )u%
}_{t-j}{\small +}\tsum\limits_{k^{\prime }=2}^{\infty }{\small \Phi }%
_{1,k^{\prime }}{\small (t-1)u}_{t-k^{\prime }}{\small +}\tsum\limits_{k^{%
\prime }=3}^{\infty }{\small \Phi }_{2,k^{\prime }}{\small (t-2)u}%
_{t-k^{\prime }}{\small +\cdots }$

$\ \ {\small +}\tsum\limits_{k^{\prime }=j+1}^{\infty }{\small \Phi }%
_{j,k^{\prime }}{\small (t-j)u}_{t-k^{\prime }}{\small +\cdots =u}_{t},$ \ $%
(A6)$

We can rewrite $(A6)$ as,%
\begin{align}
& Y_{t}+\Pi _{1}(-d_{t})u_{t-1}+\left( \Pi _{2}(-d_{t})+\Phi
_{1,2}(t-1)\right) u_{t-2}  \tag{$A7$} \\
& \text{ \ \ }+\left( \Pi _{3}(-d_{t})+\Phi _{1,3}(t-1)+\Phi
_{2,3}(t-2)\right) u_{t-3}  \notag \\
& \text{ \ \ }+\left( \Pi _{4}(-d_{t})+\Phi _{1,4}(t-1)+\Phi
_{2,4}(t-2)+\Phi _{3,4}(t-3)\right) u_{t-4}  \notag \\
& \text{ \ \ }\vdots  \notag \\
& \text{ \ \ }+\left( \Pi _{h}(-d_{t})+\Phi _{1,h}(t-1)+\Phi
_{2,h}(t-2)+\cdots +\Phi _{h-1,h}(t-h+1)\right) u_{t-h}  \notag \\
& \text{ \ \ }+\vdots  \notag
\end{align}

Let $\Pi _{h}(-d_{t})+\Phi _{1,h}(t-1)+\Phi _{2,h}(t-2)+\cdots +\Phi
_{h-1,h}(t-h+1)=\beta _{h}(t-h+1)$, for $h\geq 1$, then we can rewrite $(A7)$
as%
\begin{equation}
Y_{t}+\beta _{1}(t)u_{t-1}+\beta _{2}(t-1)u_{t-2}+\cdots +\beta
_{h}(t-h+1)u_{t-h}+\cdots =u_{t}.  \tag{$A8$}
\end{equation}

From $(A8)$, the infinite moving average representation of the process $%
y_{t} $ is%
\begin{equation}
Y_{t}=u_{t}+\sum_{j=1}^{\infty }\left( -\beta _{j}(t)\right) u_{t-j}. 
\tag{$A9$}
\end{equation}

By identification between $(A2$) and $(A9)$, we obtain%
\begin{equation}
\Psi _{j}(t)=-\beta _{j}(t-j+1).  \tag{$A10$}
\end{equation}

From $(A10)$, the first three coefficients, $\left( \Psi _{1}(t)\text{, }%
\Psi _{2}(t)\text{, }\Psi _{3}(t)\right) $ are:

\begin{itemize}
\item $\Psi _{1}(t)=-\beta _{1}(t),$

$\ \ \ \ \ \ \ \ =-\Pi _{1}(-d_{t}).$

\item $\Psi _{2}(t)=-\beta _{2}(t-1)$

\ \ \ \ \ \ \ \ $=-\left( \Pi _{2}(-d_{t})+\Phi _{1,2}(t-1)\right) $

\ \ \ \ \ \ \ \ $=-\Pi _{2}(-d_{t})-\Pi _{1}(-d_{t})\Psi _{1}(t-1)$

\ \ \ \ \ \ \ \ $=-\Pi _{2}(-d_{t})+\Pi _{1}(-d_{t})\Pi _{1}(-d_{t-1})$

\ \ \ \ \ \ \ \ $=\sum\limits_{k=1}^{2}(-1)^{k}\left( \sum\limits_{\overset{%
i_{1}=2}{i_{1}\neq 0}}\pi _{i_{1}}(-d_{t})+\sum\limits_{\underset{i_{1}\neq
0;\text{ }i_{2}\neq 0}{i_{1}+i_{2}=2}}\pi _{i_{1}}(-d_{t})\pi
_{i_{2}}(-d_{t-i_{1}})\right) $

with $i_{l}\neq 0$, $l\in \overline{1,2}$

\item $\Psi _{3}(t)=-\beta _{3}(t-2)$

\ \ \ \ \ \ \ \ $=-\left( \Pi _{3}(-d_{t})+\Phi _{1,3}(t-1)+\Phi
_{2,3}(t-2)\right) $

\ \ \ \ \ \ \ \ $=-\left( \Pi _{3}(-d_{t})+\Pi _{1}(-d_{t})\Psi
_{2}(t-1)+\Pi _{2}(-d_{t})\Psi _{1}(t-2)\right) $

\ \ \ \ \ \ \ \ $=-\Pi _{3}(-d_{t})+\Pi _{1}(-d_{t})\Pi _{2}(-d_{t-1})-\Pi
_{1}(-d_{t})\Pi _{1}(-d_{t-1})\Pi _{1}(-d_{t-2})+\Pi _{2}(-d_{t})\Pi
_{1}(-d_{t-2})$

\ \ \ \ \ \ \ \ $=\sum\limits_{k=1}^{3}(-1)^{k}\sum\limits_{i_{1}=3}\pi
_{i_{1}}(-d_{t})+\sum\limits_{i_{1}+i_{2}=3}\pi _{i_{1}}(-d_{t})\pi
_{i_{2}}(-d_{t-i_{1}})+\sum\limits_{i_{1}+i_{2}+i_{3}=3}\pi
_{i_{1}}(-d_{t})\pi _{i_{2}}(-d_{t-i_{1}})\pi _{i_{3}}(-d_{t-i_{1}-i_{2}}),$
\end{itemize}

with $i_{l}\neq 0$, $l\in \overline{1,3}.$More generally, we have,%
\begin{align}
\Psi _{j}(t)& =-\beta _{j}(t-j+1)  \notag \\
& =\left( \sum_{k=1}^{j}(-1)^{k}\sum_{i_{1}+i_{2}+\cdots +i_{k}=j}\pi
_{i_{1}}(-d_{t})\pi _{i_{2}}(-d_{t-i_{1}})\pi
_{i_{3}}(-d_{t-i_{1}-i_{2}})\cdots \pi _{i_{k}}\left( -d_{t-j+i_{k}}\right)
\right) ,  \tag{$A11$} \\
\text{with }i_{l}& \neq 0\text{, }l\in \overline{1,k}  \notag
\end{align}%
When $d_{t}=d=constant,$ we have $\left( \Psi _{1}(t)\text{, }\Psi _{2}(t)%
\text{, }\Psi _{3}(t)\right) =\left( d,\text{ }\frac{d(d+1)}{1},\text{ }%
\frac{d(d+1)(d+2)}{2}\right) $.

\end{document}